\documentclass[runningheads,a4paper]{llncs}
\usepackage{amssymb,url}
\usepackage{graphicx}
\usepackage{array,multirow} 
\usepackage{xcolor}

\urldef{\mailsa}\path|armario@us.es|
\urldef{\mailsb}\path|dane.flannery@nuigalway.ie|

\newcommand{\mod}{\, \ \mbox{mod}\, \ }

\newcommand{\abk}{\allowbreak}

\begin{document}

\mainmatter  
\title{Almost supplementary difference sets 
and quaternary sequences with optimal 
autocorrelation}

\renewcommand\rightmark{Almost supplementary
difference sets and quaternary sequences}

\author{J.~A.~Armario \and D.~L.~Flannery}

\institute{Departamento de Matem\'atica 
Aplicada I, 
Universidad de Sevilla,
Avda. Reina  Mercedes s/n, 
41012 Sevilla, Spain\\
\mailsa 
\and 
School of Mathematics, Statistics and Applied 
Mathematics, 
National University of Ireland~Galway, 
Galway H91TK33, 
Ireland\\ 
\mailsb
}
\maketitle

\begin{abstract} 
We introduce 
\emph{almost supplementary 
difference sets} (ASDS).
For odd $m$, certain ASDS in 
$\mathbb Z_m$ that have amicable 
incidence matrices
are equivalent to quaternary 
sequences of odd length 
$m$ with optimal autocorrelation.
As one consequence, if $2m-1$ 
is a prime power, or 
$m \equiv 1  \mod 4$ is prime, 
then ASDS of this kind exist. 
We also explore connections 
to optimal binary 
sequences and group cohomology.
\end{abstract}
%\keywords{Sequence, array, 
%autocorrelation, difference set, 
%cocycle, quasi-orthogonal.}

\vspace{3mm}


\section{Introduction}\label{Introduction}
A sequence  
$\phi = (\phi(0),\ldots, \phi(n-1))$ 
with all entries in $\{\pm 1\}$ or
$\{\pm 1,\pm {\rm i}\}$, where 
$\mathrm{i} = \sqrt{-1}$, 
is called {\em binary} or {\em quaternary}, 
respectively.
For a non-negative integer $w$,
the {\em periodic autocorrelation of 
$\phi$ at shift $w$} is
\begin{equation}\label{PACSequences}
R_\phi(w) =\sum_{k=0}^{n-1} 
 \phi(k)\overline{\phi(k+w)},
\end{equation}
reading arguments modulo $n$;
the overline denotes  
complex conjugate as usual. 
It is easy to see that
\begin{equation}\label{boundautocorrbinary}
  \max_{0<w<n}|R_\phi(w)|\geq 
\left\{\begin{array}{cl}
0 & \hspace{10pt} n\equiv  0  \mod 4\\
1 & \hspace{10pt} n \equiv 1 \mod 2 \\ 
2 & \hspace{10pt} n \equiv 2 \mod 4
\end{array}\right. 
\end{equation}
when $\phi$ is binary, and
\begin{equation}\label{boundautocorrquaternary}
  \max_{0<w<n}|R_\phi(w)|\geq 
\left\{\begin{array}{cl}
0 & \hspace{10pt} n \,\mbox{ even }\\
1 & \hspace{10pt} n  \, 
\mbox{ odd }
\end{array}\right. 
\end{equation}
when $\phi$ is quaternary.
A complex sequence $\phi$ such that  
$R_\phi(w) =0$ for $0<w<n$ is said to be
{\em perfect}. 
Existence of a perfect binary (resp., quaternary) 
sequence is equivalent to existence of a 
Menon-Hadamard difference set in a cyclic 
group~\cite{Jed92} 
(resp., a semi-regular relative difference
set in a cyclic group with forbidden subgroup  
of size $2$~\cite{AdL01}).
No perfect binary (resp., quaternary) sequences 
of length $n>4$ (resp., $n>16$) are known; 
see \cite{AdLM02,Schmidt00}. 
Furthermore, if $p$ is an odd prime and $s>2$
then there do not exist perfect
sequences of length $p^s$ over $p$th roots of
unity~\cite{MN09}.

Setting aside perfect sequences,
we enlarge the notion of optimality consistent
with (\ref{boundautocorrbinary}) and 
(\ref{boundautocorrquaternary});
cf.~\cite[p.~2940]{ADHKM01} and \cite{LSH03}. 
A binary sequence $\phi$ of length $n$ has
{\em optimal autocorrelation} if, for all
$w$, $0<w<n$:
\begin{enumerate}
\item[] $R_\phi(w)\in\{0,\pm 4\}$ 
\ ($n\equiv 0\mod 4$)

\vspace{1pt}

\item[] $R_\phi(w)\in\{1,-3\}$ \ 
($n\equiv 1\mod 4$)

\vspace{1pt}

\item[] 
$R_\phi(w)\in\{2,-2\}$ \ 
($n\equiv 2\mod 4$)

\vspace{1pt}

\item[] $R_\phi(w)=-1$ \ ($n\equiv 3\mod 4$).
\end{enumerate}
A quaternary sequence $\phi$  of
length $n$ has optimal 
autocorrelation---we say that $\phi$ is 
an OQS ({\em optimal quaternary sequence})---if
\begin{itemize}
\item[] 
$|R_\phi(w)|=1$ for all $w$, 
$0<w<n$ ($n$ odd)

\vspace{1pt}

\item[] $\max_{0<w<n}|R_\phi(w)|=2$ 
($n$ even). 
\end{itemize}
Actually, we will see that
if $n$ is odd and $\phi$ is an
OQS then $R_{\phi} (w)$ is real.

Table~\ref{Ex-quater-seq} records  
some existence data, extracted from 
Tables~II and IV of \cite{LSH03},
about odd length sequences 
with optimal autocorrelation. 
There are infinite families 
in all cases bar one, namely 
binary sequences $\phi$ of 
length $n\equiv 1 \mod 4$ with 
$|R_\phi(w)|=1$ for $0<w<n$.
Here examples are known only
for $n=5$ and $n=13$.
\begin{table}[h]\label{Ex-quater-seq} 
\begin{center}
\caption{Optimal sequences of 
odd length $n$ ($p$, $q$, $r$,
$q+4$, $r+2$ are prime)}
\label{tab:table1}
\begin{tabular}{|c|c|c|c|c|}\hline
\multirow{2}{*}{$n\mod 4$} & 
\multicolumn{2}{c|}{Binary} & 
\multicolumn{2}{c|}{Quaternary}\\ 
\cline{2-5}
& $\max|R_{\phi}(w)|$ 
& $n$ & $\max |R_{\phi}(w)|$ 
& $n$\\ 
\hline\hline
$1$ & $\begin{array}{c}
 1 \\[.5mm]
 3
 \end{array}$ & $\begin{array}{c}
 5,\, 13\\[.5mm] 
 p, \, q(q+4)
 \end{array}$ 
 & $1$ &
 $ \frac{p^a+1}{2}, \, p$
 \\\hline 
 $3$ & $1$ & 
 $\begin{array}{c}
 p \\
 2^a-1\\
 r(r+2)
 \end{array}$ 
 & $1$ 
 & $\frac{p^a+1}{2}$ \\ 
 \hline
\end{tabular}
 \end{center}
\end{table}

Binary sequences of length $2m$ with 
optimal `odd autocorrelation' find
practical applications in communication 
systems. 
The paper \cite{YT18} gives a procedure 
to construct such a binary sequence from an
OQS of odd length $m$. 
More is true: we demonstrate that 
these binary and quaternary sequences
are equivalent.

Optimal binary sequences 
of (even or odd) length $n$ 
may be characterized 
in terms of difference sets and almost 
difference sets in 
$\mathbb{Z}_n$; see \cite{ADHKM01}. 
A similar result for quaternary sequences 
was lacking until now. 
We explain how to characterize quaternary 
sequences of odd length $n$ with optimal 
autocorrelation as 
\emph{almost supplementary difference sets} 
in $\mathbb{Z}_n$.
 
This paper is a 
natural successor to \cite{AF17,AF18}, which
initiated the theory of 
quasi-orthogonal cocycles
and their applications in design theory.
We obtain new existence results for such cocycles 
from a connection to optimal quaternary sequences.
 
\section{Quasi-orthogonal cocycles
 and optimal sequences}

Let $G$ and $U$ be finite groups, 
with $U$ abelian.
A map $\psi:G\times G\rightarrow  U$ 
such that 
\begin{equation}\label{condiciondecociclo}
\hspace{1pt}
\psi(g,h)\psi(gh,k)=\psi(g,hk)\psi(h,k)
\quad\ \forall 
\hspace{1pt} g,h,k\in G 
\end{equation}
is a {\em cocycle} over $G$. 
The set of cocycles under pointwise 
multiplication is an abelian group,
denoted $Z^2(G,U)$. 
Given any map 
$\phi : G\rightarrow U$, the
{\em coboundary} $\partial \phi\in Z^2(G,U)$ 
is defined by 
$\partial\phi(g,h)=\phi(g)^{-1}\phi(h)^{-1}\phi(gh)$.
The coboundaries form a subgroup $B^2(G,U)$ 
of $Z^2(G,U)$. All cocycles are 
assumed to be normalized, i.e., $\psi(1,1)=1$.

We display $\psi\in Z^2(G,U)$ 
as a {\em cocyclic matrix} 
$M_\psi = \allowbreak [\psi(g,h)]_{g,h\in G}$. 
If $U = \langle -1\rangle\cong \mathbb Z_2$
and $M_\psi$ is a Hadamard matrix 
then $\psi$ is \textit{orthogonal};
in that event of course
$|G|=2$ or $|G|\equiv 0 \mod 4$.

The {\it row excess} $RE(M)$ of a cocyclic matrix $M$
indexed by $G$ is 
the sum of the absolute values of 
all row sums, apart from the
row indexed by $1_G$.
Using (\ref{condiciondecociclo}), it may be
shown that $\psi$ is orthogonal precisely when
$RE(M_\psi)$ is least, i.e.,
$RE(M_\psi)=0$. 

Henceforth we will treat mainly 
 the case $|G| \equiv 2 \mod 4$.
\begin{proposition}[{\cite[Proposition~1]{AF17}}]
Let $|G|=4t+2>2$.
If $\psi\in Z^2(G,\mathbb{Z}_2)$ then 
$RE(M_\psi)\geq \allowbreak 4t$, 
whereas $RE(M_\psi)\geq 8t+2$
if $\psi\in B^2(G,\mathbb{Z}_2)$.
\end{proposition}

By analogy with
orthogonal cocycles, 
we call $\psi$ {\em quasi-orthogonal}
if the row excess of $M_\psi$
is least possible:  
either $\psi\not \in B^2(G,\mathbb{Z}_2)$ 
and $RE(M_\psi)=4t$, or 
$\psi\in \abk B^2(G,\mathbb{Z}_2)$ and
$RE(M_\psi)=8t+2$.
The existence problem for quasi-orthogonal 
cocycles is open;
in contrast to the situation for 
orthogonal cocycles, 
we do not know of any group over which 
they do not exist.

\subsection{Generalized optimal binary arrays
and optimal quaternary sequences}

Let $G$ be the additive abelian group
${\mathbb{Z}}_{s_1}\times\cdots\times 
{\mathbb{Z}}_{s_r}$ where $s_i>1$ for 
all $i$, and put ${\bf s}=(s_1,\ldots,s_r)$. 
A (binary or quaternary) 
{\em ${\bf s}$-array} is simply a map
$\phi:G\rightarrow C$ where 
$C=\{\pm 1\}$ or $\{\pm 1,\pm \mathrm{i}\}$.
So a binary or quaternary 
sequence is an $\bf s$-array with $r=1$.

For a {\em type vector}
${\bf z}=(z_1,\ldots,z_r)\in \{ 0,1\}^r$,
let
\begin{eqnarray*}
& & E= {\mathbb{Z}}_{(z_1+1)s_1}\times
\cdots\allowbreak \times 
{\mathbb{Z}}_{(z_r+1)s_r},
\\
& & H=\{(h_1,\ldots, h_r) \in E\; |\;   
h_i=0 \ \mbox{if}\ z_i=0,\  \mbox{and} 
\ h_i=0\ \mbox{or} \ s_i\ 
\mbox{if}\ z_i=1\},
\\
& & K = \{ h\in H\; |\;  h\,
\mbox{ has even weight}\} .
\end{eqnarray*}
Then $K$ is a subgroup of the elementary
abelian $2$-subgroup $H$ of
$E$, and $E/H\cong G$. 
The {\em expansion}
of a \mbox{binary} $\bf s$-array $\phi$  
with respect to  ${\bf z}$ 
is the map $\phi'$ on $E$ defined by
\[
\phi'(x)=\left\{ 
\begin{array}{rl}
\phi(\tilde{x}) & \quad x\in \tilde{x} + K\\
-\phi(\tilde{x}) & \quad x\notin \tilde{x} +K
\end{array}\right.
\]
where $\tilde{x}$ denotes the projection of 
$x$ in $G$ (the $i$th component of $\tilde x$ 
is the $i$th component of $x$ reduced 
modulo $s_i$).

We extend the definition 
of periodic autocorrelation given in 
(\ref{PACSequences}) 
to arbitrary 
arrays $\varphi: A\rightarrow C$, 
i.e.,
\[
R_{\varphi}(a):=
\sum_{b\in A} \varphi(b)\overline{\varphi(a+b)}.
\]
A binary $\bf s$-array $\phi$ is a 
{\em generalized perfect binary array}
 (GPBA$({\bf s})$)
{\em of type ${\bf z}$} if 
\[
R_{\phi'}(x)=\abk 0 \quad \forall 
x \in E\setminus H.
\]
When ${\bf z}={\bf 0}$, this condition 
becomes $R_{\phi}( x )=\abk 0$ for all 
$x\in \abk G\setminus \{0\}$,
and if it holds then $\phi$ is a 
{\em perfect binary array}.
A GPBA$({\bf s})$ is equivalent to 
a relative difference set in 
$E/ K$ relative to $H/K$, thus
equivalent to
a cocyclic Hadamard matrix over $G$:
see \cite[Theorem~5.3]{Hug00} 
and \cite[Theorem~3.2]{Jed92}.
In particular, a binary array $\phi$
is perfect if and only if $\partial \phi$ 
is orthogonal.

Now we assume that $|G|\equiv 2 \mod 4$. 
 In particular, we assume that
$s_1/2, s_2, \abk \ldots, s_r$ are odd.
A {\em generalized optimal binary array
 of type ${\bf z}$} is
a binary $\bf s$-array $\phi$ such that
\begin{itemize}
\item[$\bullet$] 
$R_{\phi'}(x)\in
\{ 0, \pm 2 |H|\}\ \, \forall 
\hspace{.5pt} x \in E\setminus H$

\vspace{2.5pt}

\item[$\bullet$]  
$\big|\{x\in E \ | \ R_{\phi'}(x) 
= 0 \}\big|=|E|/2$ if $z_1=1$.
\end{itemize}
We write GOBA$({\bf s})$ for short.
A {\em generalized optimal binary sequence} 
(GOBS) is a GOBA($\bf s)$ with $r=z_1=1$. 

Since the abelian group $G$ does not 
have a canonical form as a direct product
of cyclic groups, the same array is a
GOBA($\bf s$) for various $\bf s$.
The following lemma reflects 
this fact (elements of
$\mathbb{Z}_2\times \mathbb{Z}_m$
and of $\mathbb{Z}_4\times \mathbb{Z}_m$
are denoted as ordered pairs;  context
will indicate which direct product is meant).
\begin{lemma}\label{gobs_goba}
Let $\varphi$ be a
binary sequence of length $2m$, $m>1$ odd. 
Define the $(2,m)$-array $\phi$ as follows.
For $m\equiv 1 \mod 4:$
\[
\phi(a,k)= \left\{\begin{array}{cl}
\varphi(k+am) &\quad 
k\equiv 0 \ \, \mathrm{mod} \ \, 4\\ 
(-1)^{1-a}\varphi(k+(1-a)m) 
& \quad
k\equiv 1 \ \, \mathrm{mod} \ \, 4\\
-\varphi(k+am) &\quad 
k\equiv 2 \ \, \mathrm{mod} \ \, 4\\ 
(-1)^a\varphi(k+(1-a)m) &\quad 
k\equiv 3 \ \, \mathrm{mod} \ \, 4
\end{array}\right. 
\]
and for $m \equiv 3  \mod 4:$ 
\[
\phi(a,k)= \left\{\begin{array}{cl}
(-1)^a\varphi(k+am) &
\quad 
k\equiv 0 \ \, \mathrm{mod} \ \, 4\\ 
\varphi(k+(1-a)m) & 
\quad
k\equiv 1 \ \, \mathrm{mod} \ \, 4\\
(-1)^{1-a}\varphi(k+am) &
\quad 
k\equiv 2 \ \, \mathrm{mod} \ \, 4\\ 
-\varphi(k+(1-a)m) &
\quad 
k\equiv 3 \ \, \mathrm{mod} \ \, 4.
\end{array}\right. 
\]
Then $\varphi$ is a GOBS 
if and only if $\phi$ is a
 GOBA$(2,m)$ of type $(1,0)$.
\end{lemma}
\begin{proof}
The identification is based on
the isomorphism 
$\mathbb{Z}_4\times 
\mathbb Z_m\rightarrow 
\mathbb Z_{4m}$ defined
by $(1,1)\mapsto 1$. 
Signs are allocated
so that $|R_{\varphi'}|$ always 
agrees with 
$|R_{\phi'}|$.
\hfill $\Box$
\end{proof}

Recall that $f: \mathbb{Z}_m\rightarrow 
\{\pm 1, \pm \mathrm{i}\}$ of odd
length $m$ is an OQS if
$|R_f(w)|=1$ for all $w$,
$1\leq w \leq m-1$.
We proceed to establish the link between 
these quaternary sequences and binary 
arrays with optimal 
autocorrelation.
\begin{remark}\label{aqbs}
There is a one-to-one correspondence between 
the set of binary $(2,m)$-arrays
$\phi$ and the set of quaternary 
sequences $f$ on $\mathbb Z_m$, given 
by
\[
f(k)=\frac{1-\mathrm{i}}{2}
(\phi(0,k)+ \mathrm{i}\phi(1,k)),
\]
\[
\phi(a,k)=\left\{\begin{array}{ll}
\mbox{Re}(f(k))-\mbox{Im}(f(k))\, 
& \quad \mbox{if } a=0\\
\mbox{Re}(f(k))+\mbox{Im}(f(k))\, 
& \quad \mbox{if } a=1
\end{array}
\right. .
\]
Translating
between additive and multiplicative versions
of $\mathbb{Z}_2\times \mathbb{Z}_2$, we 
also observe that 
\[
f(k)=\mathrm{i}^{\Phi^{-1}(\frac{1-\phi(1,k)}{2},
\frac{1-\phi(0,k)}{2})}
\]
where 
$\Phi^{-1} \colon 
\mathbb Z_2\times \mathbb Z_2\rightarrow 
\mathbb Z_4$ is the inverse Gray mapping, i.e.,
$\Phi^{-1}(0,0)=0$, $\Phi^{-1}(0,1)=1$, $\Phi^{-1}(1,1)=2$, 
and $\Phi^{-1}(1,0)=3$.   
\end{remark}
\begin{lemma}\label{lrphp} 
For $f$, $\phi$ as
in Remark{\em ~\ref{aqbs}} and $0\leq w\leq m-1$,
\[
R_f(w)=
\frac{1}{4}(R_{\phi'}(0,w)-
\mathrm{i}R_{\phi'}(1,w))
\]
where $\phi'$ is the expansion of 
$\phi$ with respect to 
${\bf z}=(1,0)$.
\end{lemma}
\begin{proof}
Routine.
\hfill $\Box$
\end{proof}

\begin{theorem} 
\label{OQSQOC}
A quaternary sequence $f$ of 
odd length $m$ is an OQS 
if and only if its corresponding
binary array $\phi$ 
is a GOBA$(2,m)$ 
of type $(1,0)$.
\end{theorem}
\begin{proof}
We have
\[
R_{\phi'}(2,w) = -R_{\phi'}(0,w), \ \,
R_{\phi'}(3,w) = -R_{\phi'}(1,w), \ \, 
\mbox{and} \ \,
R_{\phi'}(0,w)+\allowbreak R_{\phi'}(1,w)
\equiv 4 \mod 8.
\]
Thus, if $\phi$ is a GOBA$(2,m)$ 
of type $(1,0)$ then
$R_{\phi'}(0,w)$ and
$R_{\phi'}(1,w)$ cannot 
both be non-zero;
so $f$ is an OQS 
by Lemma~\ref{lrphp}.

Suppose that $f$ is an OQS.
By Lemma~\ref{lrphp}, 
again just one of $R_{\phi'}(0,w)$ 
or $R_{\phi'}(1,w)$ 
for $1\leq w \leq m-1$
is zero, while the
other is $\pm 4$. 
This also implies that the number 
of $x\in E$ such
that $R_{\phi'}(x)=0$ is 
$2m$, as required.
\hfill $\Box$
\end{proof}
 
\begin{corollary}\label{GOBS_OQS}
A quaternary sequence 
of odd length $m$ is an OQS 
if and only if the binary sequence to
which it corresponds
via Lemma{\em ~\ref{gobs_goba}} 
and Remark{\em ~\ref{aqbs}}
is a GOBS of length $2m$.
\end{corollary}

Previously, GOBS have appeared under 
other names. They are 
\emph{binary sequences with optimal
odd autocorrelation} in \cite{YT18} 
(elsewhere, `negaperiodic' 
replaces `odd').
Corollary~\ref{GOBS_OQS} furnishes
a method to construct
GOBS of length $2m$ from OQS of 
length $m$ that 
is simpler than the one in 
\cite[Construction~A, p.~389]{YT18}.
\begin{example}
Let $f=(-1,1,{\rm i},1,-{\rm i},1,{\rm i},1,-1)$. 
We calculate that 
\[
R_f= (9,-1,-1,-1, 1,1,-1,\abk -1,-1 ),
\]
so $f$ is an OQS of length $9$.
By Theorem~\ref{OQSQOC},
\[
{\footnotesize 
\left[\begin{array}{rrrrrrrrr}
-1 & \phantom{-}1 & -1 & \phantom{-}1 & 
1 & 1 & - 1 &  1 & -1\\
-1 & \phantom{-} 1 &  1 & \phantom{-}1 & 
-1 & \phantom{-} 1 & 1 & \phantom{-} 1 & -1
\end{array} 
\right]}
\]
is a GOBA($2,9$) of type $(1,0)$. 
Then by Lemma~\ref{gobs_goba}, 
\[
(-1, 1, 1, -1, 1, 1, 1, -1, 
-1, -1, -1, -1, 1, -1, -1, -1, 1, -1)
\]
is a GOBS of length $18$.
\end{example}

The next result was mentioned 
in the Introduction.
\begin{corollary}\label{vOQS}
If $f$ is an OQS
of odd length $m$ then
$R_f(w)=\pm 1$ for $1\leq w\leq m-1$.
\end{corollary}
\begin{proof}
We appeal to Lemma~\ref{lrphp} once more.
The GOBS $\varphi$ 
corresponding to $f$ has $R_{\varphi'}(u)=0$ if 
$u\in \mathbb{Z}_{4m}$ is odd;
i.e., $R_{\phi'}(1,w)=0$ for all
$w\not \equiv 0 \mod m$. \hfill $\Box$
\end{proof}

We need the next result in
Section~\ref{SectionASDS}, to prove the 
equivalence between OQS and almost supplementary
difference sets.
Denote the periodic  cross-correlation 
$\sum_{k=0}^{n-1}  a(k)b(k+w)$ of binary 
sequences $a$ and $b$ of length $n$ 
by $R_{a,b}(w)$. 
\begin{corollary}\label{relationcorrelation}
A quaternary sequence $f$ of odd length $m$
is an OQS if and only if
\[
R_{\phi(1,-)}(w)=R_{\phi(0,-)}(w)=\pm 1
\quad \mbox{and} \quad
R_{\phi(1,-),\phi(0,-)}(w)=
R_{\phi(0,-),\phi(1,-)}(w)
\] 
for $1\leq w\leq m-1$, where $\phi$ is as
in Remark{\em ~\ref{aqbs}}.
\end{corollary}
\begin{proof}
By \cite[(6)]{KS84}, we have
\[ R_f(w)=\frac{1}{2}
(R_{\phi(1,-)}(w)+R_{\phi(0,-)}(w))+
 \frac{\mathrm{i}}{2}
(R_{\phi(1,-),\phi(0,-)}(w)-
R_{\phi(0,-),\phi(1,-)}(w)).
\]
The claim is then obvious from
Corollary~\ref{vOQS}. \hfill $\Box$
\end{proof}

\subsection{Quasi-orthogonal cocycles
over ${\mathbb Z}_2\times {\mathbb Z}_m$}
\label{CocyclesOverZ2m}
We now return the discussion
to quasi-orthogonal cocycles, with a focus
on indexing group 
$G=\mathbb Z_2 \times\mathbb Z_m$, $m$ 
odd.

Define $\lambda \in Z^2(G,\langle -1\rangle)$
by
\[
\lambda((a,u),(b,w))=
\left\{\begin{array}{rl}
-1 &\quad a=b=1\\
1 & \quad \mbox{otherwise}.
\end{array}\right.
\]
Order the elements of $G$ as
$g_1=(0,0), g_2 = (0,1),\ldots,
g_{m}=(0,m-1), g_{m+1}=(1,0),
\ldots, \allowbreak g_{2m}=(1,m-1)$.
For $1\leq i\leq m$,
and with 
rows and columns indexed by $G$
under this ordering, define the
coboundary matrices $M_{\partial_i}$ 
and $M_{\partial_{i+m}}$ to be 
the respective normalizations of
\begin{equation}\label{PartialPrecursor}
\left[
\begin{array}{cc}
 C_i \ \ & J \\
 J \ \ & C_i
 \end{array}\right]\quad \mbox{and}
 \quad 
 \left[\begin{array}{cc}
 J & \  C_{i} \\
 C_{i} & \ J
 \end{array}\right], 
\end{equation}
where $C_i$ is the $m\times m$ 
back circulant $\{\pm 1\}$-matrix whose 
first row is $1$s except in position
$i$, and  $J$ is the $m\times m$
all $1$s matrix. Then 
$\{\lambda, \partial_2,\ldots,
\partial_{2m-1}\}$ is a basis 
of $Z^2(G,\langle -1\rangle)$. 

\begin{proposition}[{\cite[Theorem~2]{AF18}}]
\label{Goba-quasiorthogonal}
A normalized binary $(2,m)$-array
$\phi$ is a GOBA$(2,m)$ of type 
$(1,0)$ if and only if
$\lambda\partial\phi$ 
is quasi-orthogonal.
\end{proposition}
\begin{remark}
$\partial\phi=
\prod_{i=2}^{2m-1} 
\partial_i^{e_i}
=\partial_{2m}
\prod_{i=2}^{m} 
\partial_i^{e_i}
\cdot \, \prod_{i=m+1}^{2m-1} 
\partial_i^{1-e_i}$
where $e_i=\delta_{\phi(g_i),-1}$
(Kronecker delta).
\end{remark}
\begin{corollary}\label{exi-oqs-qoc}
There exists an OQS of length $m$
if and only if there exists 
a quasi-orthogonal cocycle over 
$\mathbb Z_2\times\mathbb Z_m$
that is not a coboundary.
\end{corollary}
\begin{proof}
Immediate from Theorem~\ref{OQSQOC}
and Proposition~\ref{Goba-quasiorthogonal}.
\hfill $\Box$
\end{proof}

\begin{remark}
Corollary~\ref{exi-oqs-qoc} and
Table~\ref{Ex-quater-seq} 
provide new infinite families of 
quasi-orthogonal  cocycles.
\end{remark}

\begin{example}
$\phi_1={\small 
\left[\begin{array}{rrr}
1 & -1 & \phantom{-} 1 \\
1 &  1 & \phantom{-} 1
\end{array}
\right]}$  
is a GOBA($2,3$), and
$\phi_2 ={\small 
\left[\begin{array}{rrrrr}
1 & -1 & \phantom{-} 1 & 
\phantom{-} 1 & \phantom{-} 1 \\
1 &  -1 & \phantom{-} 1 & 
\phantom{-} 1 & \phantom{-} 1
\end{array}
\right]}$  
is a GOBA($2,5$), both of 
type $(1,0)$.
The corresponding OQS are 
$f_1=(1, {\rm i}, 1)$ and 
$f_2=(1,-1,1,1,1)$,
with $R_{f_1} = (3,1,1)$ and 
$R_{f_2} =(5,1,1,1,1)$;  
their GOBS are 
$\varphi_1=(1,1,-1,-1,-1,1)$ and 
$\varphi_2=(1,-1,-1,-1,1,1,1,-1,1,1)$. 
The quasi-orthogonal 
cocycles $\lambda\partial \phi_i$ 
have matrices
\[
M_\lambda \circ M_{\partial_2} 
= {\footnotesize 
\left[
\renewcommand{\arraycolsep}{.05cm}
\begin{array}{rrrrrr}
1 & 1 & 1 & 1 & 1 & 1  \\
1 & 1 & -1 & -1 & -1 & -1  \\
1 & -1 & -1 & 1 & 1 & 1  \\
1 & -1 & 1 & -1 & 1 & -1  \\
1 & -1 & 1 & 1 & -1 & -1  \\
1 & -1 & 1 & -1 & -1 & 1  
\end{array}
\vspace{1pt} \right] }, 
\]  
\[
 M_\lambda\circ M_{\partial_2}\circ
 M_{\partial_7}=
 {\footnotesize \left[ 
\renewcommand{\arraycolsep}{.05cm}
\begin{array}{rrrrrrrrrr}
 1 & 1 & 1 & 1 & 1 & 1 & 1 & 1 & 1 & 1 \\
 1 & 1 & -1 & -1 & -1 & 1 & 1 & -1 & -1 & -1 \\
 1 & -1 & 1 & 1 & -1 & 1 & -1 & 1 & 1 & -1 \\
 1 & -1 & 1 & -1 & 1 & 1 & -1 & 1 & -1 & 1 \\
 1 & -1 & -1 & 1 & 1 & 1 & -1 & -1 & 1 & 1 \\
 1 & 1 & 1 & 1 & 1 & -1 & -1 & -1 & -1 & -1 \\
 1 & 1 & -1 & -1 & -1 & -1 & -1 & 1 & 1 & 1 \\
 1 & -1 & 1 & 1 & -1 & -1 & 1 & -1 & -1 & 1 \\
 1 & -1 & 1 & -1 & 1 & -1 & 1 & -1 & 1 & -1 \\
 1 & -1 & -1 & 1 & 1 & -1 & 1 & 1 & -1 & -1 
 \end{array}
 \vspace{1pt} \right] }
\]
where $\circ$ denotes Hadamard (componentwise)
product.
\end{example}

\section{Almost supplementary difference sets}
\label{SectionASDS}

Let $B=\{x_1,\ldots,x_{k_1}\}$ and 
$D=\{y_1,\ldots,y_{k_2}\}$ be
subsets of $\mathbb Z_m$. Suppose  
that the congruences
\[
x_i-x_j\equiv a\mod m,\qquad 
y_{i'}-y_{j'}\equiv a \mod m
\]
have exactly $\mu$ solutions for 
$t$ values $a\not \equiv 0 \mod m$, 
and exactly $\mu + 1$ solutions for 
the remaining $m-1-t$ values  
$a\not \equiv 0 \mod m$.
Then we call $B$ and $D$ 
{\em almost supplementary difference sets}
(ASDS); in more detail, $B$ and $D$ are 
$2$-$\{m;k_1,k_2;\mu;t\}$ ASDS. 
Clearly
\begin{equation}\label{ConstrainParameters}
k_1(k_1-1)+k_2(k_2-1)=t\mu+(m-1-t)(\mu+1),
\end{equation}
so that
$t=(m-1)(\mu+1)-k_1(k_1-1)-k_2(k_2-1)$. 
We may therefore drop  
`$t$' in the specification of the 
parameters of ASDS.
\begin{example}\label{exampleASDS1}
$B=\{1,4,5,6,7\}$ and $D=\{0,2\}$ 
are $2$-$\{9;5,2;2;2\}$ ASDS.
\end{example}
\begin{remark}
If $t=m-1$ then $B$, $D$  are 
$2$-$\{m;k_1,k_2;\mu\}$ 
{\it supplementary difference 
sets} (SDS)~\cite{KKS91}.
Another extreme is $k_1\leq 1<k_2$; 
then $D$ is an 
\textit{almost difference 
set}~\cite{ADHKM01}. 
\end{remark}

Sometimes we can obtain
ASDS from SDS by enlarging
or reducing one of the supplementary 
sets.  
This places further constraints 
on the parameters, according to 
(\ref{ConstrainParameters}).
\begin{lemma}
Suppose that $B$, $D$ are  
$2$-$\{m;k_1,k_2;\mu\}$ SDS.
\begin{itemize}
\item[{\rm (i)}]
If $B\setminus\{b\}$ for some
$b\in B$ and $D$ are 
$2$-$\{m;\allowbreak k_1-1,k_2;\mu-1; 
\frac{m-1}{2}\}$ \abk ASDS then
$k_1=(m+3)/4$ and
$\mu=(m+3)/16+(k_2^2-k_2)/(m-1)$.

\vspace{2pt}

\item[{\rm (ii)}] 
If $B\cup\{b\}$ for some
$b\in \mathbb{Z}_m\setminus (B\cup D)$
and $D$ are 
$2$-$\{m;k_1+1,k_2;\mu; \frac{m-1}{2}\}$ 
\abk ASDS then
$k_1=(m-1)/4$
and
$\mu=(m-5)/16+ (k_2^2-k_2)/(m-1)$.
\end{itemize}
\end{lemma}
\begin{example}
\begin{itemize}
    \item[(i)]
$B=\{1,2,4,8,11,16\}$,
$D=\{0,5,9,10,13,15,17,18,19,20\}$ 
are $2$-$\{21;6,\allowbreak 10;6\}$ SDS.
Also $B\setminus \{1\}$, $D$
are $2$-$\{21;5,10;5;10\}$ ASDS.

\vspace{2pt}

\item[(ii)]
$\{7,8\}$, $\{3,6,8\}$ are 
$2$-$\{9;2,3;1\}$ SDS, and
$\{4,7,8\}$, $\{3,6,8\}$ are 
$2$-$\{9;3,3;1;4\}$ ASDS.
\end{itemize}
\end{example}

Let $S$ be a subset of $\mathbb Z_m$,
with characteristic function 
$\chi_S : \mathbb Z_m\rightarrow \{0,1\}$.
Then $S^c$ will denote the (circulant)
matrix indexed by $\mathbb Z_m$ whose
$(i,j)$th entry is $1-2\chi_S(j-i)$.
We now present a 
formulation of ASDS using incidence
matrices of the supplementary sets;
this may be compared with
the Appendix of \cite{CK85}.
\begin{theorem}
\label{MatrixFormOfASDS}
\begin{enumerate}
\item[{\rm (i)}] 
Suppose that $B$ and $D$
are $2$-$\{m;k,r;\mu\}$ ASDS. 
Let $A$ be the set of all 
$a\in \mathbb{Z}_m\setminus \{ 0\}$
such that there are exactly $\mu$ 
 solutions of
\[
b-b'
\equiv a \ \, \mathrm{mod} \  \, m,
\qquad 
d-d' 
\equiv a \ \, \mathrm{mod} \ \,  m
\]
for $b, b' \in B$ and $d, d'\in D$.
Then
$[B^c(B^c)^\top+ D^c(D^c)^\top]_{i,j}$ 
is equal to
\[
[4(k+r-\mu)I_m+2(m-2(k+r-\mu))J_m]_{i,j} 
\] 
if $j-i \in A$, and 
\[
[4(k+r-\mu-1)I_m+2(m-2(k+r-\mu-1))J_m]_{i,j} 
\]
otherwise.

\vspace{2pt}

\item[{\rm (ii)}] 
Let $B^c$ and $D^c$ be $m\times m$ circulant 
$\{ \pm 1\}$-matrices such that 
$B^c(B^c)^\top+ D^c(D^c)^\top$
is as described in {\rm (i)}.
Then the subsets $B$, $D$ of
$\mathbb{Z}_m$ determined by the first 
rows of $B^c$ and $D^c$ are
$2$-$\{m;k,r;\mu \}$ ASDS,
where $k$ (resp., 
$r$) is the number of $-1$s in each row 
of $B^c$ (resp., $D^c$). 
\end{enumerate}
\end{theorem}
\begin{proof}
 (i) \,
Choose any two different rows $i$ and 
$i+a$ modulo $m$ in the concatenated
matrix $[B^c \, |\, D^c]$.
Put $\bar{\mu}=\mu$ if $a\in A$
and $\bar{\mu}= \mu+1$ if $a\notin A$.
From the definition of ASDS, we
deduce that in these two rows the column 
$[-1,-1]^\top$
appears $\bar{\mu}$ times, and
$[-1,1]^\top$, $[1,-1]^\top$ 
appear $k+r-\bar{\mu}$ times each.
Hence the inner product of the rows 
is $2m-4(k+r-\bar{\mu})$.

(ii) \, 
Since each row of $B^c$ has 
$k$ $\, -1$s and each row of $D^c$ has 
$r$ $\, -1$s,  the inner product of 
rows $i$ and $j$ of $[B^c\, |\, D^c]$
is $2m-4(k+r)+4s$ where $s$ is 
the number of columns $[-1,-1]^\top$. 
Thus, with $a \equiv j-i \mod m$,
we have $s = \bar{\mu}$ as in part (i).
\hfill $\Box$
\end{proof}

We set down a few auxiliary 
facts to prepare for  
Theorem~\ref{caracteqoasdsreplace} below.
\begin{lemma}[{\cite[Lemma~3.1]{Jed92}}]
\label{corredif}
For any array $\varphi: A\rightarrow
\{\pm 1\}$,
\[
R_\varphi(x)= |A|+ 4(d_{\varphi}(x)
-|N_{\varphi}|)
\]
where $N_{\varphi}=\{a\in A \,|\,  
\varphi(a)=-1\}$ and 
$d_{\varphi}(x)=
|N_{\varphi}\cap (x+N_{\varphi})|$.
\end{lemma}

\begin{proposition}\label{ASDSComplements}
Let $B$, $D$ be $2$-$\{m;k,r;\mu\}$ ASDS. 
Denote the complement of 
$X\subseteq \mathbb{Z}_m$ by $\overline{X}$. 
Then {\em (i)}~$\overline{B}$, $D$, 
{\em (ii)}~$B$, $\overline{D}$, 
and {\em (iii)}~$\overline{B}$, $\overline{D}$ 
are also ASDS, with parameters 
$\{m; m-k,r;m-2k+\mu\}$ in case {\em (i)}, 
$\{m;k,m-r;m-2r+\mu\}$ in case {\em (ii)}, and 
$\{m;m-k,m-r;2m-2k-2r+\mu\}$ in case {\em (iii)}. 
\end{proposition}
\begin{proof}
Write $d_{X}$ for
$d_{{\chi_{\scriptscriptstyle X}}}$.
Subsets $B$ and $D$ of $\mathbb{Z}_m$ 
are $2$-$\{m;|B|,|D|; \mu\}$ ASDS if and only
if $d_B(w)+d_D(w)=\mu$ or $\mu+1$ for all $w$,  
$1\leq w \leq m-1$.
Then the result follows from the
identity $d_{\overline{X}}(w)=m-2|X|+d_{X}(w)$.
\hfill $\Box$
\end{proof}

For the $2$-$\{m;k,r;\mu\}$ ASDS of most
interest to us, $\mu$ is determined by 
$m$, $k$, and $r$.
\begin{theorem}\label{caracteqoasdsreplace}
Let $f$ be a quaternary sequence of 
odd length $m$, with corresponding
$(2,m)$-array $\phi$ as
in Remark{\em ~\ref{aqbs}}. Then $f$ is an OQS 
if and only if 
\[
B=\{j\in \mathbb{Z}_m \; |\; \phi(0,j)=-1\},
\quad 
D=\{j\in \mathbb{Z}_m\; |\; \phi(1,j)=-1\}
\]
are $2$-$\{m;|B|,|D|;
|B|+|D|-\frac{m+1}{2}\}$ 
ASDS such that the multiset $B-D$ of 
differences $x-y$ modulo $m$ 
as $(x,y)$ ranges over $B\times D$
is symmetric, i.e., 
closed under negation.
\end{theorem}
\begin{proof}
First we deal with a technicality.
Although possibly
$|B|+|D|<\frac{m+1}{2}$,  
by Proposition~\ref{ASDSComplements}
we can take complements if 
necessary to arrange that
$|B|+|D|\geq \frac{m+1}{2}$.

By Lemma~\ref{corredif},
\[
d_{\phi(0,-)}(w)=\frac{R_{\phi(0,-)}(w)-m}{4}
+|B|
\quad \mbox{and}\quad
d_{\phi(1,-)}(w)=\frac{R_{\phi(1,-)}(w)-m}{4}
+|D|.
\]
Put $d(w)=
d_{\phi(0,-)}(w)+d_{\phi(1,-)}(w)$.
By Corollary~\ref{relationcorrelation}, 
$f$ is an OQS if and only if, 
 firstly, for $1\leq w \leq m-1$ either 
\[
d(w)=|B|+|D|-\frac{m+1}{2} \quad
\mbox{or} 
\quad
d(w)=|B|+|D|-\frac{m-1}{2} ;
\]
secondly,
\begin{equation}\label{equ2}
R_{\phi(0,-),\phi(1,-)}(w)
   =R_{\phi(0,-),\phi(1,-)}(m-w),
\end{equation}
using that $R_{b,a}(w)=R_{a,b}(n-w)$ 
for binary sequences $a$, $b$ of 
length $n$.
 
Now define  
\[
Z_l=\{(j,j+l)\in \mathbb{Z}_m\times  
\mathbb{Z}_m  \; |\; 
\phi(0,j)=\phi(1,j+l)=-1 \}
\]
and, for $X$, $Y\subseteq \mathbb Z_m$, 
\[
[X\times Y]_{w}=
\{(x,y)\in X\times Y 
\; |\; x-y\equiv  w\mod m\}.
\]
Since $R_{\phi(0,-),\phi(1,-)}(w)=
m-2(|B|+|D|-2|Z_w|)$,
the requirement (\ref{equ2}) is equivalent 
to $|Z_w|=|Z_{m-w}|$. 
We also verify that 
$|[B\times D]_{w}|=|Z_{m-w}|$ and 
$|[D\times B]_{w}|=\allowbreak |Z_w|$.
Finally, $B-D=D-B$ if and only if 
$|[B\times D]_{w}|=|[D\times B]_{w}|$
 for $1\leq \allowbreak w 
 \leq \allowbreak m-1$.
\hfill $\Box$ 
\end{proof}
\begin{example}
The $2$-$\{9; 5, 6; 6\}$
ASDS associated to the OQS 
$(1,-1,-{\rm i},-1,{\rm i},-1,\abk 
-{\rm i}, -1,1)$ of length $9$ are
$\{ 1, 3, 4, 5, 7\}$, 
$\{1, 2, 3, 5, 6, 7 \}$.
For both OQS $(1,{\rm i},1)$ 
and $(1,-1,\abk 1,1,1)$, we must take 
complements to get the ASDS
$\{1\}$, $\{0,1,2\}\subseteq 
\mathbb Z_3$ 
and $\{1\}$, $\{0,2,3,4\}\subseteq 
\mathbb Z_5$.
\end{example}

\begin{remark}\label{NewASDSfromKnownOQS}
Table~\ref{Ex-quater-seq} yields 
$2$-$\{m;|B|,|D|;|B|+|D|-\frac{m+1}{2}\}$ 
ASDS for any prime $m\equiv 1 \mod 4$ 
or $m=(p^a+1)/2$, $p$ prime.
\end{remark}

Next we state an equivalence between 
ASDS and quasi-orthogonal cocycles.
This result follows from 
Proposition~\ref{Goba-quasiorthogonal}
and Theorems~\ref{OQSQOC}
and \ref{caracteqoasdsreplace}. 
\begin{theorem}\label{caracteqoasds}
Let $\psi=\lambda\,\prod_{j=2}^{2m-1}
{\partial_j}^{k_j}$
where $k_j\in\{0,1\}$ and
$\{\lambda, \partial_2,\ldots,
\partial_{2m-1}\}$ is the basis 
of $Z^2(G,\langle -1\rangle)$ defined 
in Section{\em ~\ref{CocyclesOverZ2m}}.
Then $\psi$ is quasi-orthogonal 
if and only if
\[
B=\{j-1 \; |\; 2\leq j\leq m, \,k_j=1 \}, 
\hspace{2.5pt}
D=\{j-m-1 \; |\;  m+1\leq j\leq 2m-1, 
\,k_j=1\}
\]
are $2$-$\{m;|B|,|D|;
|B|+|D|-\frac{m+1}{2}\}$ 
ASDS such that the multiset $B-D$ of 
differences $x-y$ modulo $m$ 
as $(x,y)$ ranges over $B\times D$
is symmetric.
\end{theorem}
\begin{remark}
Since $\psi$ is normalized,
the ASDS in Theorem~\ref{caracteqoasds}
are `normalized' too ($0\notin B$). 
Also $m-1\not \in D$ 
because of the particular
basis of $Z^2(G,\langle -\rangle)$ 
chosen.
\end{remark}

\begin{example}
\begin{enumerate}
\item The ASDS in 
Example~\ref{exampleASDS1} 
satisfy the stipulations of
Theorem~\ref{caracteqoasds}, so
the cocycle
$\lambda\partial_2\partial_5
\partial_6\partial_7
\partial_8\partial_{10}\partial_{12} 
\in Z^2(\mathbb Z_2 \times \mathbb Z_9,
\langle -1\rangle)$ is quasi-orthogonal.

\vspace{2pt}

\item 
$B=\{1,2\}$ and $D=\{0,2\}$ are 
$2$-$\{7;2,2;0\}$ ASDS, but
 $B-D\neq D-B$. Hence
$\lambda\partial_2
\partial_3\partial_8
\partial_{10}
\in Z^2(\mathbb Z_2 \times 
\mathbb Z_7, \langle -1 \rangle)$
is not quasi-orthogonal.
Indeed, two rows in the lower half 
of $M_\psi$ sum to $4$.
\end{enumerate}
\end{example}
\begin{remark}
We define an equivalence relation $\sim$ on 
the set of GOBA($2,m$) by
$\phi \sim \phi'\Leftrightarrow \phi$
and $\phi'$ have the same first row and 
their second rows are negations of 
each other. Equivalence relations 
such as this carry over to compatible 
equivalence relations on 
sets of ASDS and OQS.
\end{remark}

We derive bounds on the size of
the ASDS in Theorem~\ref{caracteqoasds}.
\begin{corollary} 
Suppose that $B$ and $D$ are
$2$-$\{m;k,r;k+r-\frac{m+1}{2}\}$ 
ASDS where $m$ is odd,
$0\not \in B$, 
and $m-1\not \in D$. Then
\[
\frac{(m-1)^2}{2}\leq 
(k+r)m-(k^2+r^2)\leq \frac{m^2-1}{2} .
\]
\end{corollary}
\begin{proof}
There exists 
$\psi= \prod_{i\in B} \partial_{i+1} 
\prod_{i\in D}\partial_{i+m+1}$ 
such that the number of $-1$s in 
row $j$ of $M_\psi$ 
 for $2\leq j\leq m$
is $m \pm 1$.
Alternatively, counting in $M_\psi$ 
before row normalization reveals that
the total number of $\, -1$s in 
these rows is 
$2k(m-k)+\abk 2r(m-r)$. 
The inequalities follow by
comparing the counts.
 \hfill $\Box$
\end{proof}

Our ultimate result 
is an accompaniment to
Theorem~\ref{MatrixFormOfASDS}.
\begin{lemma}
For any nonempty subsets $B$, $D$ of 
$\mathbb Z_m$, the multiset $B-D$ 
is symmetric if and only if 
$B^c$ and $D^c$ are amicable, 
i.e., $B^c (D^c)^\top$
is symmetric.
\end{lemma}
\begin{proof}
Note that $B^c(D^c)^\top$ and 
$D^c(B^c)^\top$ are circulant.
If $u=(u_0,u_1,\ldots, \abk u_{m-1})$ 
and $v=(v_0,v_1,\ldots,v_{m-1})$ are 
the first rows of $B^c(D^c)^\top$ 
and $D^c(B^c)^\top$, then
\[
u_0=v_0,\,\, u_1=v_{m-1},\, \ldots,\,\, u_i=v_{m-i},\,\ldots,\,\,  u_{m-1}=v_1.
\]
Consequently $B^c(D^c)^\top=
D^c(B^c)^\top$ if and only if
\[
u_1=u_{m-1}, \
u_2=u_{m-2}, \
\ldots , \
u_{\frac{m-1}{2}}=u_{\frac{m+1}{2}} .
\]

Let $(b_0,b_1,\ldots,b_{m-1})$ and
$(d_0,d_1,\ldots,d_{m-1})$ be the 
respective first rows of $B^c$ and 
$D^c$. Then
\begin{eqnarray*}
& & u_i =b_ 0d_{[-i]_m}+b_1d_{[1-i]_m}
+ \cdots  + b_{m-1}d_{[-1-i]_m} \\
& & u_{m-i}=d_ 0b_{[-i]_m}+d_1b_{[1-i]_m}
+ \cdots + d_{m-1}b_{[-1-i]_m}.
\end{eqnarray*}
We check that
$u_i = u_{m-i}$ if and
only if the number of summands
$b_jd_k$ in $u_i$ with 
$b_j=d_k = -1$ is equal to the
number of summands $b_{j'}d_{k'}$ in
$u_{m-i}$ with
$b_{j'}=\abk d_{k'} = -1$.
Since $j-k\equiv i \equiv k'-j' \mod m$,
the proof is complete.
\hfill $\Box$
\end{proof}

In conclusion, and with reference to 
Remark~\ref{NewASDSfromKnownOQS} and the 
existence problem for quasi-orthogonal cocycles, 
we pose  the open problem 
of constructing new OQS from 
new ASDS (cf.~the construction 
in \cite{ADHKM01} of 
optimal binary sequences from almost 
difference sets).

\subsubsection*{Acknowledgments.} 
The first author thanks Kristeen Cheng for 
reading the manuscript, and V\'ictor \'Alvarez 
for assistance with computations.
This research was partially supported by project 
FQM-016 funded by JJAA (Spain).

\end{document}